\documentclass[11pt,a4paper]{amsart}
\usepackage{latexsym}
\usepackage{graphicx}
\usepackage{subfig}
\usepackage{caption}
\usepackage{float}
\usepackage{enumerate}
\usepackage[top=3.2cm,bottom=3.8cm,left=3cm,right=2cm]{geometry}
\usepackage{mathrsfs}
\usepackage{amssymb}
\usepackage{amsbsy}
\usepackage[colorlinks,linkcolor=blue,citecolor=blue,pagebackref]{hyperref}

\newtheorem{theorem}{Theorem}[section]
\newtheorem{lemma}[theorem]{Lemma}

\theoremstyle{definition}

\theoremstyle{remark}
\newtheorem{remark}[theorem]{Remark}

\numberwithin{equation}{section}

\begin{document}
	\begin{center}
		{\large\bf The Brunn-Minkowski inequality for the generalized Gaussian distribution}
	\end{center}

	\vskip 15pt
\begin{center}
	{\small\bf  Ge\ Xiong\ \  \ \ \ \ Kai-Wen\ Yang}\\~~ \\
	\small{School of Mathematical Sciences, Key Laboratory of Intelligent Computing and Applications (Ministry of Education), Tongji University, Shanghai, 200092, China}
\end{center}

\vskip 5pt
\footnotetext{E-mail addresses: 1. xiongge@tongji.edu.cn;\ 2. yangkaiwen@tongji.edu.cn.}
\footnotetext{Research of the authors was supported by NSFC No. 12271407.}
	
	\begin{center}
		\begin{minipage}{14cm}
			{{\bf Abstract:} 
				Let $\mu_p$ be the generalized Gaussian distribution on $\mathbb{R}^n$ with density $e^{-\frac{|x|^p}{p}}$ multiplied by a constant depending on $p\ge 1$ and $n$, and $\alpha_p(n)$ be the largest number such that the Brunn-Minkowski type inequality
				$$
				\mu_p(\lambda K+(1-\lambda) L)^{\alpha_p(n)} \geq \lambda \mu_p(K)^{\alpha_p(n)}+(1-\lambda) \mu_p(L)^{\alpha_p(n)}
				$$ holds for all convex bodies $K,L$ in $\mathbb{R}^n$ containing the origin and $\lambda\in[0,1]$. 
                In this paper, the new lower and upper bounds for $\alpha_p(n)$ are found, and their asymptotically optimality as $n\to +\infty$ is proved.}
			
			
		\end{minipage}
	\end{center}

	\vskip 20pt
	\section{\bf Introduction}
	\vskip 5pt
	
	
	The setting for this paper is the $n$-dimensional Euclidean space, $\mathbb{R}^n$. Let $\mu$ be a finite Borel measure on $\mathbb{R}^n$. It is said that $\mu$ satisfies the \emph{Brunn-Minkowski inequality} with exponent $\alpha$, if for all Borel sets $K, L \subseteq \mathbb{R}^n$  and  $\lambda \in [0,1]$,
	\begin{equation*}
		\mu(\lambda K+(1-\lambda) L)^\alpha \geq \lambda \mu(K)^\alpha+(1-\lambda) \mu(L)^\alpha.
	\end{equation*}
	Here, $\lambda K+(1-\lambda) L=\left\{\lambda x+(1-\lambda) y: x \in K, y \in L\right\}$ is the \emph{Minkowski combination} of $K$ and $L$. 
	By the Hölder inequality, if $\alpha_1>\alpha_2$ and the inequality holds for $\alpha_1$, then it holds for $\alpha_2$. It suggests that the inequality becomes \emph{stronger} as $\alpha$ increases. 
    
    The Brunn-Minkowski inequality has connections with many other fundamental inequalities, such as the isoperimetric inequality, the Prékopa-Leindler inequality, the Sobolev inequality and the Brascamp-Lieb inequality, etc. See, e.g., Barthe \cite{Barthe} and Bobkov and Ledoux \cite{Bobkov1,Bobkov2} for details. 
	
	The classical Brunn–Minkowski inequality reads: If $K, L \subseteq \mathbb{R}^n$ are Borel sets and $\lambda \in [0,1]$, then for the $n$-dimensional Lebesgue measure $V_n$,
	$$
	V_n(\lambda K+(1-\lambda) L)^\frac{1}{n} \geq \lambda V_n(K)^\frac{1}{n}+ (1-\lambda) V_n(L)^\frac{1}{n}.
	$$
    It is not only one of the cornerstones in convex geometry, but also a powerful tool in many areas such as probability and statistics,  information theory and physics. Please refer to the excellent survey by Gardner \cite{gar1}.
	
	Let $\mu$ be a finite Borel measure on $\mathbb{R}^n$ with density $f$ with respect to the Lebesgue measure on $\mathbb{R}^n$, and $\alpha\in [-\infty,\frac{1}{n}]$. In the 1970s, Borell \cite{borell2, borell} proved that  $\mu$ satisfies the Brunn-Minkowski inequality with exponent $\alpha$, if and only if for all $x, y \in \mathbb{R}^n$ with $f(x), f(y)>0$ and  all $\lambda\in[0,1]$,
	$$
	f(\lambda x+(1-\lambda) y)^\beta \geq \lambda f(x)^\beta+(1-\lambda) f(y)^\beta,
	$$
	where $\frac{1}{\alpha}=\frac{1}{\beta}+n$. The cases where $\alpha$ or $\beta$ equals to $-\infty$, $0$ or $+\infty$ are interpreted in the limiting sense.
	Specifically,  if $\alpha=\beta=0$, then the Borell theorem reduces to the celebrated Prékopa theorem \cite{AP}: Let $\mu$ be a Borel measure on $\mathbb{R}^n$ with density $f$ with respect to the Lebesgue measure on $\mathbb{R}^n$. Then $\mu$ is \emph{log-concave}, i.e., for all Borel sets $K$, $L$ and  $\lambda\in [0,1]$,
	$$
	\mu(\lambda K+(1-\lambda) L) \geq \mu(K)^\lambda \mu(L)^{1-\lambda}
	$$
	if and only if its density $f$ is of the form $e^{-V}$, where the potential $V:\mathbb{R}^n\to \mathbb{R}$ is convex.
	
	The Borell theorem accurately quantifies the  exponent  in the Brunn-Minkowski inequality for measure $\mu$ over \emph{all} Borel sets. It is remarkable that  if  the measure $\mu$ is defined on  sets with more \emph{geometric flavors}, for example, the class of  origin-symmetric convex bodies in $\mathbb{R}^n$, the class of convex bodies containing the origin in $\mathbb{R}^n$, or the class of star bodies in $\mathbb{R}^n$, a possible ``better" exponent may exist. 
	
	Recently, that looking for a better exponent  in the Brunn-Minkowski inequality on some special classes has become a central topic in the Brunn–Minkowski theory. Several fascinating results and conjectures are appeared. See, e.g., \cite{ Aishwarya,rotem,Rotem,Moschidis,Livshyts,milman1,Livshyts2,Z2,zhang1,yang}. Moreover, this topic is connected \cite{milman3,Z2,Saroglou2,Saroglou} to the longstanding log-Brunn-Minkowski conjecture, which is posed by Böröczky, Lutwak, Yang and Zhang \cite{log} and states that the  volume of the geometric mean of two symmetric convex sets is bounded below by the geometric mean of their volumes.
    
    Let $\mu$ be a measure on $\mathbb{R}^n$ with $d\mu=e^{-V}dx$ and that $V:\mathbb{R}^n\to \mathbb{R}$ is convex and even.
	Livshyts, Marsiglietti, Nayar and Zvavitch \cite[Proposition 1]{Z2} showed that if the log-Brunn-Minkowski conjecture is confirmed, then for all origin-symmetric convex bodies $K, L$ in $\mathbb{R}^n$ and $\lambda \in[0,1]$, 
	\begin{equation}\label{BMeven}
		\mu(\lambda K+(1-\lambda) L)^{\frac{1}{n}} \geq \lambda \mu(K)^{\frac{1}{n}}+(1-\lambda) \mu(L)^{\frac{1}{n}} .
	\end{equation}
	In 2023, Cordero-Erausquin and Rotem \cite{Rotem} proved that \eqref{BMeven} is true if in addition  the measure $\mu$ is rotationally invariant. In \cite{Livshyts2}, Livshyts proved \eqref{BMeven} type inequality with exponent $\frac{1}{n^{4+o(1)}}$.
	
	\vskip 3pt
	
	In 2010, Gardner and Zvavitch \cite{Gardner} asked whether the standard Gaussian distribution $\gamma_n$, i.e., the measure with density $\frac{d \gamma_n}{d x}=\frac{1}{(2 \pi)^{n / 2}} e^{-\frac{|x|^2}{2}}$, satisfies the Brunn-Minkowski inequality with exponent $\frac{1}{n}$ on the class of convex bodies containing the origin? 
	In 2013, a \emph{counterexample} to this problem was constructed by Nayar and Tkocz  \cite{Nayar}.

	In 2021,  Kolesnikov and Livshyts \cite{Livshyts} proved that if $K$ and $L$ are  convex bodies \emph{containing the origin} in $\mathbb{R}^n$ and  $\lambda \in[0,1]$, then
	$$
	\gamma_n((\lambda K+(1-\lambda) L)^{\frac{1}{2n}} \geq \lambda \gamma_n(K)^{\frac{1}{2n}}+(1-\lambda) \gamma_n(L)^{\frac{1}{2n}} .$$ 
    Later, Eskenazis and Moschidis \cite{Moschidis} settled the Gardner-Zvavitch problem when  $K$ and $L$ are \emph{origin-symmetric} convex bodies: If $K$ and $L$ are origin-symmetric convex bodies in $\mathbb{R}^n$ and  $\lambda \in[0,1]$, then
	\begin{equation}\label{gaussSys}
		\gamma_n(\lambda K+(1-\lambda) L)^{\frac{1}{n}} \geq \lambda \gamma_n(K)^{\frac{1}{n}}+(1-\lambda) \gamma_n(L)^{\frac{1}{n}} .
	\end{equation}
	In 2025, Aishwarya and Li \cite{Aishwarya}  gave an alternative proof of \eqref{gaussSys} by the optimal transport method, and established a corresponding entropy power inequality. Till now, the largest exponent for $\gamma_n$ on the class of convex bodies containing the origin  remains \emph{unknown}.
	
	In 2026, Aishwarya and Rotem \cite{rotem} proved that for all \emph{star bodies} $K$, $L$ in $\mathbb{R}^n$ and  $\lambda \in[0,1]$,
	\begin{equation}\label{pqici}
		\mu(\lambda K+(1-\lambda) L)^{\frac{p-1}{pn}} \geq \lambda \mu(K)^{\frac{p-1}{pn}}+(1-\lambda) \mu(L)^{\frac{p-1}{pn}},
	\end{equation} where $\mu$ is a log-concave measure on $\mathbb{R}^n$ with density $e^{-V}$ and $V$ is $p$-homogeneous, $p\in (1,+\infty)$.  A function $V:\mathbb{R}^n\to \mathbb{R}$ is \emph{$p$-homogeneous}, if $V(\lambda x)=\lambda^p V(x)$ for all $x \in \mathbb{R}^n$ and  $\lambda>0$.

    \vskip 3pt
	
	In this article, we focus on establishing the Brunn-Minkowski inequality for the \emph{generalized Gaussian distribution} $\mu_p$ on $\mathbb{R}^n$, $p\ge 1$,   with
	$$
	d\mu_p=\frac{\Gamma(\frac{n}{2})}{2 \pi^{\frac{n}{2}} p^{\frac{n}{p}-1} \Gamma(\frac{n}{p})}e^{-\frac{|x|^p}{p}}dx
	$$
    over the class of convex bodies containing the origin. 
    
	The generalized Gaussian distribution $\mu_p$ is a natural generalization of the standard Gauss distribution $\gamma_n$ with $d \gamma_n=\frac{1}{(2 \pi)^{n / 2}} e^{-\frac{|x|^2}{2}}dx$, and connects several important distributions. If $p=2$, it is precisely the standard Gaussian measure; If $p=1$, it is closely related to the multivariate Laplace distribution; If $p\to +\infty$, $\mu_p$ tends to the uniform distribution on the Euclidean unit ball $B^n$ in $\mathbb{R}^{n}$. Throughout this article, assume $n\ge 2$  and  $p\ge 1$.

	Denote the \emph{upper incomplete gamma function} as  $\Gamma(s, x) = \int_{x}^{+\infty} t^{s-1} e^{-t} dt$, $s\in \mathbb{R}$, $x> 0$. Please refer to, e.g.,  \cite[6.5.3]{handbook} on this function.  Our main result is the following.
	
	\begin{theorem}\label{main}
		If $\mu_p$ is the generalized Gaussian distribution on $\mathbb{R}^n$ and $\alpha_p(n)$ is the largest number such that the inequality
		$$
		\mu_p(\lambda K+(1-\lambda) L)^{\alpha_p(n)} \geq \lambda \mu_p(K)^{\alpha_p(n)}+(1-\lambda) \mu_p(L)^{\alpha_p(n)} 
		$$ holds for all convex bodies $K,L$ in $\mathbb{R}^n$ containing the origin and   $\lambda\in[0,1]$, then
		$$
		\frac{1}{n}e^{\frac{(p-1)n}{p}}( \frac{(p-1)n}{p} )^{\frac{n}{p}} \Gamma(1-\frac{n}{p}, \frac{(p-1)n}{p}) \le \alpha_p(n)\le 1 - \frac{p}{n-1} \frac{\Gamma(\frac{n+p-2}{p}) \Gamma(\frac{n}{p})}{\Gamma(\frac{n-1}{p})^2}.
		$$
	\end{theorem}
	
	It is interesting that if $p=2$,  Theorem \ref{main} gives that
	$$
	\frac{1}{n}e^{\frac{n}{2}} ( \frac{n}{2} )^{\frac{n}{2}} \Gamma(1-\frac{n}{2}, \frac{n}{2}) \le \alpha_2(n)\le 1 - \frac{2}{n-1} \frac{ \Gamma(\frac{n}{2})^2}{\Gamma(\frac{n-1}{2})^2}.
	$$
	If in addition $n=2,3,4,\ldots$, a direct numerical computation yields that $$\alpha_2(2)\in [0.298, 0.363],\ \ \alpha_2(3)\in [0.189, 0.215],\ \ \alpha_2(4)\in [0.138, 0.152],\ldots.$$
	In \cite{Nayar}, Nayar and Tkocz showed that $\alpha_2(2)<0.5$. In fact, using their arguments, one can obtain $\alpha_2(2)\le 1-\frac{2}{\pi}\approx 0.363$.
	In \cite[Remark 6.7]{Livshyts},  Kolesnikov and Livshyts obtained $\alpha_2(2)\ge 0.25$, and they wish to push the lower bound to $0.298$. See Remark \ref{re3} for details. 
	
	If $p=1$ and $n\ge 2$, the lower bound in Theorem \ref{main} is interpreted in the limiting sense and equals zero. Indeed, since
 $$\Gamma(1-\frac{n}{p}, \frac{(p-1) n}{p})=\int_{\frac{(p-1) n}{p}}^{+\infty} t^{-\frac{n}{p}} e^{-t} d t \leq \int_{\frac{(p-1) n}{p}}^{+\infty} t^{-\frac{n}{p}} d t=\frac{(\frac{(p-1) n}{p})^{1-\frac{n}{p}}}{\frac{n}{p}-1},$$ 
it follows that
 $$0\le \lim_{p\to 1^+}\frac{1}{n}e^{\frac{(p-1)n}{p}}( \frac{(p-1)n}{p} )^{\frac{n}{p}} \Gamma(1-\frac{n}{p}, \frac{(p-1)n}{p})\le \lim_{p\to 1^+}\frac{1}{n} e^{\frac{(p-1) n}{p}} \frac{\frac{(p-1) n}{p}}{\frac{n}{p}-1}=0.$$ 
 Combining the upper bound in Theorem \ref{main} with $p=1$, it immediately gives that $\alpha_1(n)\equiv 0$. In contrast, on the class of \emph{origin-symmetric} convex bodies,  Livshyts \cite{Livshyts2} established \eqref{BMeven} type inequality with  exponent $\frac{1}{n^{4+o(1)}}>0$ for even log-concave measures.

 It is pointed out that for fixed  $p\ge 1$,  our bounds are \emph{asymptotically optimal} in the following sense: as \(n\to +\infty\), the lower and upper bounds are both equal to \(\frac{p-1}{np} + o\!\left(\frac{1}{n}\right)\). For given $n$, as $p\to +\infty$, the lower and upper bounds both go to $\frac{1}{n}$, which recovers the classical Brunn-Minkowski inequality for convex bodies. One can refer to Sections \ref{sub4.1} and \ref{sub4.2} for details.

	In \cite{rotem}, the authors proved that on the class of \emph{star bodies} the exponent $\alpha\ge \frac{p-1}{pn}$, if  the measure $\mu$ on $\mathbb{R}^n$ involved in the Brunn-Minkowski inequality satisfies that $d\mu=e^{-V}dx$ and  $V:\mathbb{R}^n\to \mathbb{R}$ is a $p$-homogeneous convex function, $p\in (1,+\infty)$.   It is interesting that if restricted on the class of \emph{convex bodies containing the origin}, Theorem \ref{qici} tells us that for such measure $\mu$, 
	$$
	\alpha\ge \frac{1}{n}e^{\frac{(p-1)n}{p}} ( \frac{(p-1)n}{p} )^{\frac{n}{p}} \Gamma(1-\frac{n}{p}, \frac{(p-1)n}{p})>\frac{p-1}{pn}.$$
     See Remark \ref{re2} for details. However, since our approach  tightly relies on the injectivity of Gauss mapping, we have to admit that the \emph{convexity} assumption cannot be removed.
	
	To prove the lower bound in Theorem \ref{main}, we follow the route opened by Kolesnikov and Livshyts \cite{Livshyts}, and prove a new estimate (Lemma \ref{gen}) for the exponent involved in the Brunn-Minkowski inequality. Using this estimate, we develop an analysis technique to attack the  task  by virtue of of  \emph{radial functions}. See Lemma \ref{origin} for details. Inspired by the counterexample provided by Nayar and Tkocz \cite{Nayar}, we prove the upper bound in Theorem \ref{main}.
	
	This article is organized as follows. In Section 2,  some basic facts are introduced. Theorem \ref{main} is proved in Section 3. The  asymptotic for the bounds as $n\to +\infty$ or $p\to +\infty$ are provided in Section 4. In Section 5, we extend the lower bound of generalized Gaussian distribution in Theorem \ref{main} to measures with $p$-homogeneous convex potentials.  Finally,  our bounds and the known bounds are illustrated by using Python 3.14.0.

	\vskip 20pt
	\section{\bf Preliminaries}
	\vskip 5pt
	
	For quick later reference, we collect some basic facts here. Good references on convex bodies are the books by Gardner \cite{Gardner1} and Schneider \cite{Sch}.
	
	As usual, let $\vert x \vert$ be the standard Euclidean norm of $x$  and $\left\langle x, y \right\rangle$  be the
	standard inner product of $x$ and $y$ in $\mathbb{R}^n$.  Let $\mathbb{S}^{n-1}=\{u\in \mathbb{R}^{n}:|u|=1\}$
	be the boundary of the Euclidean unit ball $B^n$.  The volume of $B^n$ is $\omega_n=\frac{\pi^{\frac{n}{2}}}{\Gamma(\frac{n}{2}+1)}$.
	Write  $B^n(R)$ for the centered ball with radius $R$. 
	
	A \emph{convex body} is a compact convex body that has a nonempty interior. Write \(\mathcal{K}^n_o\) and \(\mathcal{K}^n_{os}\) for the class of convex bodies containing  the origin and origin-symmetric convex bodies in $\mathbb{R}^n$, respectively.
	 A set $K\subset \mathbb{R}^n$ is a star body, if it is a Borel set satisfying $\lambda x\in K$ for every $x\in K$ and $\lambda\in [0,1]$. The \emph{radial function} $\rho_K$ of a star body \( K \) is defined by
	$$
	\rho_K(u) = \sup\{\lambda \geq 0 : \lambda x \in K\},\quad u\in\mathbb{S}^{n-1}.
	$$
	
	Let $V:\mathbb{R}^n\to \mathbb{R}$ be a convex function. Then $V$ is locally Lipschitzian, $\nabla V$ is continuous almost everywhere, and $\nabla^2 V$  exists and is positive semi-definite almost everywhere. If $V$ is strictly convex, then $\nabla V$ is monotone  on its domain. However, one cannot still expect $\nabla^2 V$ to be positive definite in this case. A classic example is the following: enumerate all rational numbers in the interval $[0,1]$ as $r_1, r_2, r_3, \ldots$, and define
	$$
	g(x)=\sum_{n=1}^{+\infty} \frac{1}{2^n} \cdot 1_{\left[r_n, 1\right]}(x),\quad V(x)=\int_0^x g(t) d t.
	$$
	Then $V(x)$ is a strict convex function on $[0,1]$, but $V''(x)=0$ almost everywhere. 

    Let $A$ be an invertible $n\times n$ matrix and  $u,v
\in
\mathbb{R}^n$ such that $1+v^{T}A^{-1}u\neq 0$. The \emph{Sherman--Morrison} formula states that
    \[(A+uv^{T})^{-1}=A^{-1}-\frac{A^{-1}uv^{T}A^{-1}}{1+v^{T}A^{-1}u}.\]

	The following theorem shown by Kolesnikov and Livshyts \cite[Lemma 2.3]{Livshyts} is needed.
	
	\begin{theorem}[\cite{Livshyts}]\label{KL}
		Given a family $\mathcal{K}$ of convex bodies in $\mathbb{R}^n$ that is closed under Minkowski addition. If $\mu$ is a Borel measure on $\mathbb{R}^n$ with density $e^{-V}$ and that $V:\mathbb{R}^n\to \mathbb{R}$ is convex, and $\alpha$ is the largest number such that the inequality
		$$
		\mu(\lambda K+(1-\lambda) L)^{\alpha} \geq \lambda \mu(K)^{\alpha}+(1-\lambda) \mu(L)^{\alpha}
		$$ holds for all convex bodies $K,L\in \mathcal{K}$ and $\lambda\in[0,1]$, then
		\begin{equation}
			\alpha\ge \inf_{K \in \mathcal{K}}\frac{1}{\mu(K)}\int_K\big(\left\|\nabla^2 u\right\|^2+\left\langle\nabla^2 V \nabla u, \nabla u\right\rangle\big) d \mu,
		\end{equation}
		where $u$ is a $C^2$ function on $K$ with $\Delta u-\left\langle \nabla V, \nabla u\right\rangle=1_K$.
	\end{theorem}
	This formulation builds upon the machinery developed by Colesanti \cite{Colesanti} and Kolesnikov and Milman \cite{milman2,milman1}. Recall that Colesanti \cite{Colesanti} translated the Brunn-Minkowski inequality into an equivalent Poincaré-type inequality on	the boundary of convex bodies; Kolesnikov and Milman \cite{milman2,milman1} expressed a function on $\partial K$ as the Neumann data of an other function defined in the interior of $K$, and then used the Bochner method to prove the Poincaré-type inequality.
	
	\vskip 20pt
	
	\section{\bf Proof of the main results}
	
	\vskip 5pt
	
	In this section, we present the proof of Theorem \ref{main}.
	
	\subsection{The lower bound}
	\ 
    
	First, we give a new estimate for the exponent  $\alpha$ in Theorem \ref{KL}.
	
	\begin{lemma}\label{gen}
		If $\mu$ is a finite Borel  measure on $\mathbb{R}^n$ with density $e^{-V}$ where $V:\mathbb{R}^n\to \mathbb{R}$ is convex and $\nabla^2V$ is positive definite almost everywhere,  and $\alpha$ is the largest number such that 
		$$
		\mu(\lambda K+(1-\lambda) L)^{\alpha} \geq \lambda \mu(K)^{\alpha}+(1-\lambda) \mu(L)^{\alpha}
		$$ holds for all convex bodies $K,L$ in $\mathbb{R}^n$ containing the origin and $\lambda\in[0,1]$, then
		$$
		\alpha\ge \inf_{K \in \mathcal{K}_o^n} \frac{1}{n\mu(K)}\int_K \frac{n}{n+\langle\left(\nabla^2 V\right)^{-1} \nabla V, \nabla V\rangle }d \mu.
		$$
	\end{lemma}
	
	\begin{proof}
		By Theorem \ref{KL}, it suffices to show that
		$$
		\int_K(\left\|\nabla^2 u\right\|^2+\left\langle\nabla^2 V \nabla u, \nabla u\right\rangle) d \mu \geq \frac{1}{n}\int_K \frac{n}{n+\langle\left(\nabla^2 V\right)^{-1} \nabla V, \nabla V\rangle }d \mu,
		$$
		where $u$ is a $C^2$ function on $K$ with $\Delta u-\left\langle \nabla V, \nabla u\right\rangle=1_K$.

		Since $\nabla^2 V$ is positive definite almost everywhere, by the Cauchy-Schwarz inequality we have
		$$
		\left\langle\nabla^2 V \nabla u, \nabla u\right\rangle\langle\left(\nabla^2 V\right)^{-1} \nabla V, \nabla V\rangle\ge \left\langle \nabla V, \nabla u\right\rangle^2.
		$$
		Thus,
		\begin{equation*}
			\int_K \left\langle\nabla^2 V \nabla u, \nabla u\right\rangle d \mu\ge \int_K \frac{\left\langle \nabla V, \nabla u\right\rangle^2}{\langle\left(\nabla^2 V\right)^{-1} \nabla V, \nabla V\rangle}d \mu.
		\end{equation*}
		
		Meanwhile, using the Cauchy-Schwarz inequality, we have
		$$
		\left\|\nabla^2 u\right\|^2 \geq \frac{1}{n} |\Delta u|^2.
		$$
		Thus, 
		\begin{equation*}
			\begin{aligned}
				\int_K\left\|\nabla^2 u\right\|^2 d \mu &\geq \frac{1}{n} \int_K|\Delta u|^2 d \mu  = \frac{1}{n} \int_K|\left\langle \nabla V, \nabla u\right\rangle+1|^2 d \mu \\&=\frac{1}{n}\int_K\big( 1+2\left\langle \nabla V, \nabla u\right\rangle+\left\langle \nabla V, \nabla u\right\rangle^2\big) d \mu.
			\end{aligned}
		\end{equation*}
		
		Adding the above two  inequalities,  we have
		$$
		\begin{aligned}
			&\int_K\big(\left\|\nabla^2 u\right\|^2+\left\langle\nabla^2 V \nabla u, \nabla u\right\rangle\big) d \mu \\ \geq& \frac{1}{n}\int_K \Big( 1+2\left\langle \nabla V, \nabla u\right\rangle+\big(1+\frac{n}{\langle\left(\nabla^2 V\right)^{-1} \nabla V, \nabla V\rangle}\big)\left\langle \nabla V, \nabla u\right\rangle^2\Big) d \mu\\
			=&\frac{1}{n}\int_K \Big( \big(1+\frac{n}{\langle\left(\nabla^2 V\right)^{-1} \nabla V, \nabla V\rangle}\big)\big(\left\langle \nabla V, \nabla u\right\rangle+\frac{\langle\left(\nabla^2 V\right)^{-1} \nabla V, \nabla V\rangle}{n+\langle\left(\nabla^2 V\right)^{-1} \nabla V, \nabla V\rangle}\big)^2\\
            &+\frac{n}{n+\langle\left(\nabla^2 V\right)^{-1} \nabla V, \nabla V\rangle}\Big) d \mu \ge \frac{1}{n}\int_K \frac{n}{n+\langle\left(\nabla^2 V\right)^{-1} \nabla V, \nabla V\rangle }d \mu,
		\end{aligned}
		$$
		as desired.
	\end{proof}

    To reformulate  $\inf_{K \in \mathcal{K}^n_o} \frac{1}{n\mu(K)}\int_K \frac{n}{n+\langle\left(\nabla^2 V\right)^{-1} \nabla V, \nabla V\rangle }d \mu$, we prove the following lemma.
	
	\begin{lemma}\label{origin}
		Let $\mu$ be a finite Borel measure on $\mathbb{R}^n$ with density $e^{-V}$ and that $V:\mathbb{R}^n\to \mathbb{R}$ is continuous. If $F:\mathbb{R}^{n}\to (0,+\infty)$ is continuous, then
		$$
		\inf_{K \in \mathcal{K}^n_o}\frac{1}{\mu(K)}\int_K Fd\mu=\inf_{\theta\in \mathbb{S}^{n-1}}\{\inf_{R>0}\frac{\int_{0}^{R}F(r\theta)r^{n-1}e^{-V(r\theta)}dr}{\int_{0}^{R}r^{n-1}e^{-V(r\theta)}dr}\}.
		$$
	\end{lemma}
	
	\begin{proof}
		Integrating by polar coordinates, we have
		$$
		\begin{aligned}
			\frac{1}{\mu(K)}\int_K Fd\mu=&\frac{\int_K Fd\mu }{\int_K d\mu}
			=&\frac{\int_{\mathbb{S}^{n-1}}\int_{0}^{\rho_K(\theta)} F(r\theta)r^{n-1}e^{-V(r\theta)}dr d\theta}{\int_{\mathbb{S}^{n-1}}\int_{0}^{\rho_K(\theta)} r^{n-1} e^{-V(r\theta)}dr d\theta}.
		\end{aligned}
		$$
		Define $$g_\theta(t)=\frac{\int_{0}^{t} F(r\theta) r^{n-1} e^{-V(r\theta)}dr}{\int_{0}^{t} r^{n-1} e^{-V(r\theta)}dr},\quad t> 0,~~\theta\in \mathbb{S}^{n-1}.$$ 
		
		First,  we prove $\inf_{K \in \mathcal{K}^n_o}\frac{1}{\mu(K)}\int_K Fd\mu\ge \inf_{\theta\in \mathbb{S}^{n-1}}\{\inf_{R>0}\frac{\int_{0}^{R}F(r\theta)r^{n-1}e^{-V(r\theta)}dr}{\int_{0}^{R}r^{n-1}e^{-V(r\theta)}dr}\}$.
		
		Since $g_\theta(\rho_K(\theta))\ge \inf_{R>0}g_{\theta}(R)$ for all $\theta\in \mathbb{S}^{n-1}$, we have
		$$
		\int_{0}^{\rho_K(\theta)} F(r\theta)r^{n-1} e^{-V(r\theta)}dr\ge  \inf_{\theta\in \mathbb{S}^{n-1}}\{\inf_{R>0}g_{\theta}(R)\} \int_{0}^{\rho_K(\theta)} r^{n-1}e^{-V(r\theta)}dr, \quad \forall \theta\in \mathbb{S}^{n-1}.
		$$
		Integrating both sides over $\mathbb{S}^{n-1}$, we obtain
		$$
		\frac{\int_{\mathbb{S}^{n-1}}\int_{0}^{\rho_K(\theta)} F(r\theta)r^{n-1}e^{-V(r\theta)}dr d\theta}{\int_{\mathbb{S}^{n-1}}\int_{0}^{\rho_K(\theta)} r^{n-1} e^{-V(r\theta)}dr d\theta}\ge \inf_{\theta\in \mathbb{S}^{n-1}}\{\inf_{R>0}g_{\theta}(R)\}.
		$$
		Thus,
		$$
		\inf_{K \in \mathcal{K}^n_o}\frac{1}{\mu(K)}\int_K Fd\mu\ge \inf_{\theta\in \mathbb{S}^{n-1}}\{\inf_{R>0}g_{\theta}(R)\}=\inf_{\theta\in \mathbb{S}^{n-1}}\{\inf_{R>0}\frac{\int_{0}^{R}F(r\theta)r^{n-1}e^{-V(r\theta)}dr}{\int_{0}^{R}r^{n-1}e^{-V(r\theta)}dr}\}.
		$$
		
		Second, we prove $\inf_{K \in \mathcal{K}^n_o}\frac{1}{\mu(K)}\int_K Fd\mu\le \inf_{\theta\in \mathbb{S}^{n-1}}\{\inf_{R>0}\frac{\int_{0}^{R}F(r\theta)r^{n-1}e^{-V(r\theta)}dr}{\int_{0}^{R}r^{n-1}e^{-V(r\theta)}dr}\}$.
		
		By the definition of the infimum, for each  $\varepsilon>0$, there exist $\theta_\varepsilon\in \mathbb{S}^{n-1}$ and $R_\varepsilon>0$ such that $g_{\theta_\varepsilon}(R_\varepsilon)<\inf_{\theta\in \mathbb{S}^{n-1}}\{\inf_{R>0}g_{\theta}(R)\}+\varepsilon$. Given $\delta>0$ and $R>0$. Let $$K_{\delta,R}=\{(r,\theta)|\ 0\le r\le R, \|\theta-\theta_\varepsilon\|\le \delta, \theta\in\mathbb{S}^{n-1} \}.$$ Then $K_{\delta,R}$ is a convex body, and
		$$
		\begin{aligned}
			\frac{1}{\mu(K_{\delta,R})}\int_{K_{\delta,R}} Fd\mu=\frac{\int_{\|\theta-\theta_\varepsilon\|\le \delta, \theta\in\mathbb{S}^{n-1}}\int_{0}^{R} F(r\theta)r^{n-1}e^{-V(r\theta)}dr d\theta}{\int_{\|\theta-\theta_\varepsilon\|\le \delta, \theta\in\mathbb{S}^{n-1}}\int_{0}^{R} r^{n-1}e^{-V(r\theta)}dr d\theta}.
		\end{aligned}
		$$
		Since  $F$ and $V$ are continous, for $\delta\to 0$, we have 
		$$
		\frac{\int_{0}^{R} F(r\theta)r^{n-1}e^{-V(r\theta)}dr }{\int_{0}^{R} r^{n-1}e^{-V(r\theta)}dr }\to \frac{\int_{0}^{R} F(r\theta_\varepsilon)r^{n-1}e^{-V(r\theta_\varepsilon)}dr}{\int_{0}^{R} r^{n-1}e^{-V(r\theta_\varepsilon)}dr }
		$$
		for all $ \theta\in\mathbb{S}^{n-1}$:  $\|\theta-\theta_\varepsilon\|\le \delta$. Thus, $$\lim_{\delta\to 0}\frac{1}{\mu(K_{\delta,R})}\int_{K_{\delta,R}} Fd\mu=g_{\theta_\varepsilon}(R).$$
		
		Letting $R$ tends to $R_\varepsilon$, we have $$
		\lim_{R\to R_\varepsilon}\lim_{\delta\to 0}\frac{1}{\mu(K_{\delta,R})}\int_{K_{\delta,R}} Fd\mu=g_{\theta_\varepsilon}(R_\varepsilon)<\inf_{\theta\in \mathbb{S}^{n-1}}\{\inf_{R>0}g_{\theta}(R)\}+\varepsilon.
		$$
		By the arbitrariness of $\varepsilon$, we have 
		$$
		\inf_{K \in \mathcal{K}^n_o}\frac{1}{\mu(K)}\int_K Fd\mu\le \inf_{\theta\in \mathbb{S}^{n-1}}\{\inf_{R>0}g_{\theta}(R)\}=\inf_{\theta\in \mathbb{S}^{n-1}}\{\inf_{R>0}\frac{\int_{0}^{R}F(r\theta)r^{n-1}e^{-V(r\theta)}dr}{\int_{0}^{R}r^{n-1}e^{-V(r\theta)}dr}\}.
		$$
		This completes the proof.
	\end{proof}

	\begin{lemma}\label{dandiao}
		Let $\mu$ be a finite Borel measure on $\mathbb{R}^n$ with density $e^{-V}$ and that $V:\mathbb{R}^n\to \mathbb{R}$ is continuous. If $F:\mathbb{R}^{n}\to (0,+\infty)$ is  continuous and $F(r\theta)$ is decreasing in $r\ge 0$ for all  $\theta\in \mathbb{S}^{n-1}$, then
		$$
		\inf_{K \in \mathcal{K}^n_o}\frac{1}{\mu(K)}\int_K Fd\mu=\inf_{\theta\in \mathbb{S}^{n-1}}\{\frac{\int_{0}^{+\infty}F(r\theta)r^{n-1}e^{-V(r\theta)}dr}{\int_{0}^{+\infty}r^{n-1}e^{-V(r\theta)}dr}\}.
		$$
	\end{lemma}
	
	\begin{proof}
		Define $$g_\theta(t)=\frac{\int_{0}^{t} F(r\theta)r^{n-1}e^{-V(r\theta)}dr}{\int_{0}^{t} r^{n-1}e^{-V(r\theta)}dr},\quad t> 0,\, \theta\in \mathbb{S}^{n-1}.$$
	 Since $F(r\theta)$ is decreasing in $r\ge 0$ for all  $\theta\in \mathbb{S}^{n-1}$, we have
		$$
		\begin{aligned}
			\frac{d}{d t} g_\theta(t)&=\frac{e^{-V(t \theta)} t^{n-1}F(t \theta) \int_0^t r^{n-1} e^{-V(r \theta)} d r-e^{-V(t \theta)} t^{n-1} \int_0^tF(r \theta)r^{n-1} e^{-V(r \theta)}  d r}{\big(\int_0^t r^{n-1} e^{-V(r \theta)} d r\big)^2}\\
			&=\frac{e^{-V(t \theta)} t^{n-1} \int_0^t(F(t \theta)-F(r \theta))r^{n-1} e^{-V(r \theta)} d r}{\big(\int_0^t r^{n-1} e^{-V(r \theta)} d r\big)^2}\le 0 .
		\end{aligned}
		$$
		Hence, $g_\theta(t)$ is decreasing. Since $F>0$, it follows that $g_\theta(t)>0$. So,  
		$$
		\inf_{R>0}g_\theta(R)=g_\theta(+\infty).
		$$
		By Lemma \ref{origin}, we have
		$$
		\begin{aligned}
			\inf_{K \in \mathcal{K}^n_o}\frac{1}{\mu(K)}\int_K Fd\mu&= \inf_{\theta\in \mathbb{S}^{n-1}}\{\inf_{R>0}g_{\theta}(R)\}=\inf_{\theta\in \mathbb{S}^{n-1}}\{g_{\theta}(+\infty)\}\\&=\inf_{\theta\in \mathbb{S}^{n-1}}\{\frac{\int_{0}^{+\infty}F(r\theta)r^{n-1}e^{-V(r\theta)}dr}{\int_{0}^{+\infty}r^{n-1}e^{-V(r\theta)}dr}\}.
		\end{aligned}
		$$
		This completes the proof.
	\end{proof}

	\begin{proof}[Proof of the lower bound:]
		For brevity, write $\mu$ for the measure on $\mathbb{R}^n$ with density $e^{-\frac{|x|^p}{p}}$. 
		Up to a constant depending on $n$ and $p$, it suffices  to prove the lower bound for $\mu$.
		We divide the proof into three steps. We first assume that $p>1$, then use  approximation  to prove the case $p=1$.
		
		\textbf{Step 1.} Reformulate the lower bound of $\alpha_p(n)$ for $p>1$.
		
		Define $V: \mathbb{R}^n\to R$ as $V(x)=\frac{|x|^p}{p}$, $p>1$. Then,
		$$
		\begin{aligned}
			\nabla V(x)&= |x|^{p-2}x, \\
			\nabla^2 V(x)&=(p-1)|x|^{p-4} x x^T+|x|^{p-2}\big(I-\frac{x x^T}{|x|^2}\big),\quad x\neq 0,
		\end{aligned}
		$$
        where $x^T$ denotes the transpose of $x \in  \mathbb{R}^n$.
		
        Since $(p-1)|x|^{p-4} x x^T$ is positive definite and $I-\frac{x x^T}{|x|^2}$ is semi-positive definite for $x\neq 0$, it follows that $\nabla^2 V(x)$ is positive definite for $x\neq 0$.  Hence
		$$
		(\nabla^2 V)^{-1}(x)= \frac{x x^T}{(p-1)|x|^{p-4}}+\frac{1}{|x|^{p-2}}(I-\frac{x x^T}{|x|^2}),\quad x\neq 0,
		$$
		and therefore
		$$
		\langle(\nabla^2 V)^{-1} \nabla V, \nabla V\rangle (x)= \frac{|x|^p}{p-1}, \quad x\neq 0.
		$$
        
		Thus, for all $r> 0$ and $\theta \in \mathbb{S}^{n-1}$, we obtain
		$$
		\frac{n}{n+\langle(\nabla^2 V)^{-1} \nabla V, \nabla V\rangle }(r\theta)= \frac{n}{n+\frac{r^p}{p-1}}, 
		$$
		which is continuous, positive and decreasing in $r$ for all  $\theta\in\mathbb{S}^{n-1}$. By Lemmas \ref{gen} and  \ref{dandiao}, 
		$$
		\begin{aligned}
			\alpha_p(n)\ge&\inf_{K \in \mathcal{K}^n_o} \frac{1}{n\mu(K)}\int_K \frac{n}{n+\langle(\nabla^2 V)^{-1} \nabla V, \nabla V\rangle }d \mu\\=&\frac{1}{n}\inf_{\theta\in \mathbb{S}^{n-1}}\{\frac{\int_{0}^{+\infty}\frac{n}{n+\frac{r^p}{p-1}}r^{n-1}e^{-V(r\theta)}dr}{\int_{0}^{+\infty}r^{n-1}e^{-V(r\theta)}dr}\}=\frac{1}{n}\inf_{\theta\in \mathbb{S}^{n-1}}\{\frac{\int_{0}^{+\infty}\frac{ne^{-\frac{r^p}{p}}}{n+\frac{r^p}{p-1}}r^{n-1}dr}{\int_{0}^{+\infty}e^{-\frac{r^p}{p}}r^{n-1}dr}\}.
		\end{aligned}
		$$
		
		\textbf{Step 2.}  Compute the above infimum.

        Let $u = \frac{r^p}{p}$, $p>1$. Then
\begin{equation*}
    \frac{\int_{0}^{+\infty}\frac{ne^{-\frac{r^p}{p}}}{n+\frac{r^p}{p-1}}r^{n-1}dr}{\int_{0}^{+\infty}r^{n-1}e^{-\frac{r^p}{p}}dr}
    = \frac{\frac{n}{p} p^{-\frac{n}{p}} \int_0^{+\infty} \frac{e^{-u} u^{\frac{n}{p}-1}}{n+\frac{p}{p-1} u} du}{\frac{1}{p} p^{-\frac{n}{p}} \Gamma\bigl(\frac{n}{p}\bigr)}
    = \frac{n}{\Gamma(\frac{n}{p})} \int_0^{+\infty} \frac{e^{-u} u^{\frac{n}{p}-1}}{n+\frac{p}{p-1} u} du.
\end{equation*}

Note that
\[
\begin{aligned}
    &\int_0^{+\infty} \frac{e^{-u} u^{\frac{n}{p}-1}}{n+\frac{p}{p-1} u} du = \frac{p-1}{p} \int_0^{+\infty} \frac{e^{-u} u^{\frac{n}{p}-1}}{\frac{(p-1)n}{p} + u} du\\
    =&\frac{p-1}{p} \int_0^{+\infty} e^{-u} u^{\frac{n}{p}-1} ( \int_0^{+\infty} e^{-(\frac{(p-1)n}{p} + u)s} ds ) du.
\end{aligned}
\]
Using the Fubini Theorem, we have
\[
\begin{aligned}
     \int_0^{+\infty} \frac{e^{-u} u^{\frac{n}{p}-1}}{n+\frac{p}{p-1} u} du=& \frac{p-1}{p} \int_0^{+\infty} e^{-\frac{(p-1)n}{p} s} ( \int_0^{+\infty} e^{-u(1+s)} u^{\frac{n}{p}-1} du ) ds \\
    =& \frac{p-1}{p} \Gamma(\frac{n}{p}) \int_0^{+\infty} e^{-\frac{(p-1)n}{p} s} (1+s)^{-\frac{n}{p}} ds.
\end{aligned}
\]

Let $t = \frac{(p-1)n}{p} (1+s)$, $p>1$. It follows that
\[
\begin{aligned}
    \int_0^{+\infty} e^{-\frac{(p-1)n}{p} s} (1+s)^{-\frac{n}{p}} ds &= \int_{\frac{(p-1)n}{p}}^{+\infty} e^{-(t - \frac{(p-1)n}{p})}  (\frac{p}{n(p-1)})^{-\frac{n}{p}+1} t^{-\frac{n}{p}} dt \\
    &= (\frac{(p-1)n}{p})^{\frac{n}{p} - 1} e^{\frac{(p-1)n}{p}} \int_{\frac{(p-1)n}{p}}^{+\infty} e^{-t} t^{-\frac{n}{p}} dt \\
    &= e^{\frac{(p-1)n}{p}} (\frac{(p-1)n}{p})^{\frac{n}{p} - 1} \Gamma(1-\frac{n}{p}, \frac{(p-1)n}{p}).
\end{aligned}
\]

Thus
\[
\begin{aligned}
    \frac{n}{\Gamma(\frac{n}{p})} \int_0^{+\infty} \frac{e^{-u} u^{\frac{n}{p}-1}}{n+\frac{p}{p-1} u} du &= \frac{n}{\Gamma(\frac{n}{p})} \cdot \frac{p-1}{p} \Gamma(\frac{n}{p})  \cdot e^{\frac{(p-1)n}{p}} (\frac{(p-1)n}{p})^{\frac{n}{p} - 1} \Gamma(1-\frac{n}{p}, \frac{(p-1)n}{p}) \\
    &= e^{\frac{(p-1)n}{p}} ( \frac{(p-1)n}{p} )^{\frac{n}{p}} \Gamma(1-\frac{n}{p}, \frac{(p-1)n}{p}).
\end{aligned}
\]

	Combining Step 1 and the above equations, we obtain
		$$
		\begin{aligned}
			\alpha_p(n)&\ge\frac{1}{n}\inf_{\theta\in \mathbb{S}^{n-1}}\{\frac{\int_{0}^{+\infty}\frac{ne^{-\frac{r^p}{p}}}{n+\frac{r^p}{p-1}}r^{n-1}dr}{\int_{0}^{+\infty}e^{-\frac{r^p}{p}}r^{n-1}dr}\}\\&=\frac{1}{n}e^{\frac{(p-1)n}{p}} ( \frac{(p-1)n}{p} )^{\frac{n}{p}} \Gamma(1-\frac{n}{p}, \frac{(p-1)n}{p}),\quad p>1.
		\end{aligned}
		$$

		\textbf{Step 3.} Prove the case $p=1$.

        Let $$
		m(p) = \frac{1}{n}e^{\frac{(p-1)n}{p}} ( \frac{(p-1)n}{p} )^{\frac{n}{p}} \Gamma(1-\frac{n}{p}, \frac{(p-1)n}{p}), \quad p>1.
		$$
         Since  $\lim_{p\to 1^+}( \frac{(p-1)n}{p} )^{\frac{n}{p}} \Gamma(1-\frac{n}{p}, \frac{(p-1)n}{p})=0$, it follows that $\lim_{p \to 1^+} m(p) = 0$. 
		
        Meanwhile, from Step 2 it follows that for all $K, L \in \mathcal{K}_o^n$ and all $\lambda \in [0, 1],$
		\begin{equation*}
			\mu_{p}(\lambda K+(1-\lambda) L)^{m(p)} \geq \lambda \mu_{p}(K)^{m(p)}+(1-\lambda) \mu_{p}(L)^{m(p)}.
		\end{equation*}
		So,
		\begin{equation}\label{logmu}
			\frac{\mu_p(\lambda K+(1-\lambda) L)^{m(p)} - 1}{m(p)} \geq \lambda \frac{\mu_p(K)^{m(p)} - 1}{m(p)}+(1-\lambda) \frac{\mu_p(L)^{m(p)} - 1}{m(p)}.
		\end{equation}

 Since $\lim_{p \to 1^+}e^{-\frac{|x|^p}{p}}=e^{-|x|}$ and $e^{-|x|}\le 1$ for  $x\in \mathbb{R}^n$, by the dominated convergence theorem,  we have $\lim_{p \to 1^+} \mu_p(A) = \mu_1(A)$ for any  $A \in \mathcal{K}_o^n$.
 
		Taking  $p \to 1^+$ on both sides of \eqref{logmu}, it yields that
		$$
		\ln \mu_1(\lambda K+(1-\lambda) L) \geq \lambda \ln \mu_1(K)+(1-\lambda) \ln \mu_1(L).
		$$
		That is, $\alpha_1(n) \geq 0=m(1)$.
	\end{proof}

	\begin{remark}\label{re3}
		For the standard Gaussian distribution $\gamma_n$, our result implies that
		$$
		\alpha_2(n)\ge \frac{1}{n}e^{\frac{n}{2}} ( \frac{n}{2} )^{\frac{n}{2}} \Gamma(1-\frac{n}{2}, \frac{n}{2}).
		$$
		For $n=2$, it gives that $ \alpha_2(2)\ge \frac{e}{2} \Gamma\left(0, 1\right)\approx 0.298$.
		
		In \cite[Remark 6.7]{Livshyts}, Kolesnikov and Livshyts demonstrated that
		\begin{equation}\label{6.7}
			\alpha_2(2)\ge\inf_{K \in \mathcal{K}^2_o} \frac{1}{\gamma_2(K)}\int_{K}\frac{1}{|x|^2+2}d\gamma_2.
		\end{equation}
		Using Caffarelli’s contraction theorem \cite{Caffarelli}, they  showed that the infimum in \eqref{6.7} equals $\frac{e}{2} \Gamma\left(0, 1\right)$ on the origin-symmetric  set $\mathcal{K}^2_{os}$. Using our method, we showed that this infimum equals $\frac{e}{2} \Gamma\left(0, 1\right)$ on the set $\mathcal{K}^2_{o}$. Moreover, we extend this result to arbitrary dimension and $p\ge 1$.
	\end{remark}
	
	\subsection{The upper bound}
	
	\begin{proof}[Proof of the upper bound:] For brevity, write $\mu$ for the measure on $\mathbb{R}^n$ with density $e^{-\frac{|x|^p}{p}}$, $p\ge 1$. 
		Up to a constant depending on $n$ and $p$, it suffices  to prove the  upper bound for $\mu$. 
		
		\textbf{Step 1.} Construct convex sets.
		
		Fix $\alpha \in(0, \frac{\pi}{2})$ and $\varepsilon>0$. Let
		$$
		\begin{aligned}
			A& =\big\{\left(x_1,x_2,\ldots, x_n\right)\in \mathbb{R}^n:\, x_n \geq \big(\sum_{i=1}^{n-1}x_i^2\big)^{\frac{1}{2}} \tan \alpha\big\}. \\
			B& =\big\{\left(x_1,x_2,\ldots, x_n\right)\in \mathbb{R}^n:\, x_n \geq \big(\sum_{i=1}^{n-1}x_i^2\big)^{\frac{1}{2}} \tan \alpha-\varepsilon\big\}.
		\end{aligned}
		$$
		Then $o \in A \cap B$, $A, B$ are convex, and $\lambda A+(1-\lambda)B=A+(1-\lambda)(0,\ldots,0,-\varepsilon)$ for all $\lambda\in [0,1]$. 
		
		\textbf{Step 2.} Evaluate the measure of the Minkowski combination of $A$ and $B$.
		
		Define
		$$
		H(r, s)=\int_s^{+\infty} e^{-\frac{(r^2+t^2)^{p/2}}{p}} d t, \quad r \geq 0,~~s\in \mathbb{R}.
		$$
        
		By the polar coordinates on span$\{x_1,\ldots,x_{n-1}\}$, we have
		\begin{equation*}
			\begin{aligned}
				\mu(A)=& \int_{\mathbb{R}^{n-1}}\int_{(\sum_{i=1}^{n-1}x_i^2)^{\frac{1}{2}}\tan \alpha}^{+\infty} e^{-\frac{|x|^{p}}{p}} d x_n dx_1dx_2\cdots dx_{n-1} \\
                =&(n-1)\omega_{n-1}\int_{0}^{+\infty}r^{n-2}\int_{r\tan \alpha}^{+\infty}e^{-\frac{(r^2+x_n^2)^{p/2}}{p}} d x_n dr
                \\=&(n-1)\omega_{n-1}\int_{0}^{+\infty}r^{n-2}H(r,r\tan \alpha)dr.
                	\end{aligned}
		\end{equation*}
        
        Similarly, we have
        \begin{equation*}
            \mu(B)=(n-1)\omega_{n-1}\int_{0}^{+\infty}r^{n-2}H(r,r\tan \alpha-\varepsilon)dr
		\end{equation*}
and
                    \begin{equation*}
				\mu(\lambda A+(1-\lambda) B)=(n-1)\omega_{n-1}\int_{0}^{+\infty}r^{n-2}H(r,r\tan \alpha-(1-\lambda)\varepsilon)dr.
		\end{equation*}
		
		\textbf{Step 3.} Compute the expansion of  $\mu(\lambda A+(1-\lambda) B)^q-\lambda\mu(A)^q-(1-\lambda)\mu(B)^q$, $q>0$.
		
		Since
		$$
		\begin{aligned}
			H(r, r \tan\alpha)&=\int_{r \tan\alpha}^{+\infty} e^{-\frac{(r^2+t^2)^{p/2}}{p}} d t,\\
			\frac{\partial H}{\partial s}\Big|_{(r,r \tan \alpha)}&=-e^{-\frac{(r^2+(r \tan \alpha)^2)^{p/2}}{p}}=-e^{-\frac{r^p}{p \cos^p \alpha}},\\
			\frac{\partial^2 H}{\partial s^2}\Big|_{(r,r \tan \alpha)}&=\frac{d}{ds} \Big(-e^{-\frac{(r^2+s^2)^{p/2}}{p}}\Big)\Big|_{(r,r \tan \alpha)}=\frac{\sin \alpha r^{p-1}}{\cos^{p-1} \alpha} e^{-\frac{r^p}{p \cos^p \alpha}},
		\end{aligned}
		$$
		it follows that
		\begin{equation*}
			H(r,r \tan \alpha-\varepsilon)=H(r,r \tan \alpha)+e^{-\frac{r^p}{p \cos^p \alpha}} \varepsilon+\frac{1}{2} \frac{\sin \alpha}{\cos^{p-1} \alpha} r^{p-1} e^{-\frac{r^p}{p \cos^p \alpha}} \varepsilon^2+o(\varepsilon^2).
		\end{equation*}
        
		Let
		\begin{equation*}
			\begin{aligned}
				I_0&=\int_0^{+\infty} r^{n-2} H(r, r \tan \alpha) d r,\\
				I_1&=\int_0^{+\infty} r^{n-2} e^{-\frac{r^p}{p \cos^p \alpha}} d r =\cos^{n-1}\alpha \int_0^{+\infty} t^{n-2} e^{-\frac{t^p}{p}} d t,\\
				I_2&=\int_0^{+\infty} r^{n+p-3} \frac{\sin \alpha}{\cos^{p-1} \alpha} e^{-\frac{r^p}{p \cos^p \alpha}} d r = \sin\alpha \cos^{n-1}\alpha\int_0^{+\infty} t^{n+p-3} e^{-\frac{t^p}{p}} d t.
			\end{aligned}
		\end{equation*}
        
		By  Step 2,  we have
		\begin{equation*}
			\begin{aligned}
				\mu(A)&=(n-1)\omega_{n-1}I_0:=M,\\
				\mu(B)&=M(1+\frac{I_1}{I_0} \varepsilon+ \frac{I_2}{2I_0} \varepsilon^2+o(\varepsilon^2)),\\
				\mu(\lambda A+(1-\lambda) B)&=M(1+(1-\lambda)\frac{I_1}{I_0} \varepsilon+(1-\lambda)^2\frac{I_2}{2I_0} \varepsilon^2+o(\varepsilon^2)).
			\end{aligned}
		\end{equation*}
		
		For given $q > 0$, since 
		\begin{equation*}
			\begin{aligned}
				\mu(B)^q&=M^q(1+ q\frac{I_1}{I_0} \varepsilon+\frac{q}{2}( \frac{I_2}{I_0}+ (q-1)\frac{I_1^2}{I_0^2}) \varepsilon^2+o(\varepsilon^2)),\\
				\mu(\lambda A+(1-\lambda) B)^q&=M^q(1+ q(1-\lambda)\frac{I_1}{I_0} \varepsilon+\frac{q}{2}( \frac{I_2}{I_0}+ (q-1)\frac{I_1^2}{I_0^2})(1-\lambda)^2 \varepsilon^2+o(\varepsilon^2)),
			\end{aligned}
		\end{equation*}
		it follows that
		$$
		\mu(\lambda A+(1-\lambda) B)^q-\lambda\mu(A)^q-(1-\lambda)\mu(B)^q= -\frac{q}{2}M^q ( \frac{I_2}{I_0}+ (q-1)\frac{I_1^2}{I_0^2}) \lambda(1-\lambda) \varepsilon^2+o(\varepsilon^2).
		$$
		
		\textbf{Step 4.} Analyze the sign of the  coefficients of $\varepsilon^2$.
		
		By the coarea formula, we have
		$$
		\begin{aligned}
			I_0&=\frac{\mu(A)}{(n-1)\omega_{n-1}}=\frac{1}{(n-1)\omega_{n-1}}\int_A e^{-\frac{|x|^p}{p}}dx=  \frac{1}{(n-1)\omega_{n-1}}\int_{0}^{+\infty}\int_{\partial B(r) \cap A} e^{-\frac{r^p}{p}} d \mathcal{H}^{n-1}  dr\\
            &=\int_{0}^{+\infty}\int_0^{\frac{\pi}{2}-\alpha} e^{-\frac{r^p}{p}} \rho^{n-1}\sin ^{n-2} \theta d \theta dr  =c_0\int_0^{\frac{\pi}{2}-\alpha} \sin ^{n-2} \theta d \theta,
            \end{aligned}
		$$
        where $c_0= \int_0^{+\infty} t^{n-1} e^{-\frac{t^p}{p}} dt=p^{\frac{n}{p}-1} \Gamma\big(\frac{n}{p}\big)$.
        
        By Step 3, we have
        $I_1=c_1 \cos^{n-1}\alpha, $ where $c_1=\int_0^{+\infty} t^{n-2} e^{-\frac{t^p}{p}} d t= p^{\frac{n-1}{p}-1} \Gamma(\frac{n-1}{p})$; and $I_2= c_2\sin\alpha \cos^{n-1}\alpha,$ where $c_2 = \int_0^{+\infty} t^{n+p-3} e^{-\frac{t^p}{p}} d t=p^{\frac{n+p-2}{p}-1} \Gamma(\frac{n+p-2}{p})$.
		
		Let $\beta=\frac{\pi}{2}-\alpha$. If $\alpha\to \frac{\pi}{2}^-$, i.e., $\beta\to 0^+$, then
		$$
		\begin{aligned}
        \int_0^{\frac{\pi}{2}-\alpha} \sin ^{n-2} \theta d \theta&=\int_0^{\beta} \theta^{n-2}(1+o(1)) d\theta=\frac{\beta^{n-1}}{n-1}+o(\beta^{n-1}),\\
			\cos^{n-1} \alpha&=\sin^{n-1} \beta=\beta^{n-1}+o(\beta^{n-1}),\\
			\sin\alpha \cos^{n-1} \alpha&=\cos\beta \sin^{n-1} \beta=\beta^{n-1}+o(\beta^{n-1}).
		\end{aligned}
		$$
		So $\frac{I_1}{I_0} = (n-1)\frac{c_1}{c_0} + o(1)$ and $\frac{I_2}{I_0} = (n-1)\frac{c_2}{c_0} + o(1)$, as $\alpha\to \frac{\pi}{2}^-$. Hence
		$$
		\frac{I_2}{I_0} + (q-1)\frac{I_1^2}{I_0^2} = (n-1)\frac{c_2}{c_0} + (q-1)(n-1)^2\frac{c_1^2}{c_0^2} + o(1), \quad \text{as}\ \alpha\to \frac{\pi}{2}^-.
		$$
        
		Assume 	$(n-1)\frac{c_2}{c_0} + (q-1)(n-1)^2\frac{c_1^2}{c_0^2}>0,$ i.e., $q > 1 - \frac{c_0c_2}{(n-1) c_1^2}.$ Since $q>0$, for all $\lambda\in (0,1)$ we have $$-\frac{q}{2}M^q ( \frac{I_2}{I_0}+ (q-1)\frac{I_1^2}{I_0^2}) \lambda(1-\lambda)<0,\quad \text{as}\ \alpha\to\frac{\pi}{2}^-.$$  
        Let $\lambda=\frac{1}{2}$, we have $$\mu(\frac{1}{2} A+\frac{1}{2} B)^q<\frac{1}{2}\mu(A)^q+\frac{1}{2}\mu(B)^q,\quad  \text{as}\ \varepsilon\to 0\ \text{and}\ \alpha\to\frac{\pi}{2}^-.$$ 
        
        Let $A_R=A\,\cap B^n(R)$ and $B_R=B\, \cap B^n(R)$. Then, $$\lim_{R \to +\infty}\mu(A_R)=\mu(A), \ \lim_{R \to +\infty}\mu(B_R)=\mu(B),\ \text{and}\ \lim_{R \to +\infty}\mu(\frac{1}{2}A_R+\frac{1}{2}B_R)=\mu(\frac{1}{2}A+\frac{1}{2}B).$$ 
		Thus
		$$
		\mu(\frac{1}{2} A_R+\frac{1}{2} B_R)^q<\frac{1}{2}\mu(A_R)^q+\frac{1}{2}\mu(B_R)^q, \quad \text{as} \  \varepsilon\to 0,\ \alpha\to\frac{\pi}{2}^-\  \text{and}\ R\to +\infty,
		$$
		which suggests that
		\begin{equation*}
			\alpha_p(n)\le 1 - \frac{c_0c_2}{(n-1) c_1^2}=1 - \frac{p}{n-1} \frac{\Gamma\big(\frac{n}{p}\big) \Gamma\big(\frac{n+p-2}{p}\big)}{\Gamma\big(\frac{n-1}{p}\big)^2}.
		\end{equation*}
		This completes the proof.
	\end{proof}
	
	In particular, if $p=1$, $\mu_1$ is the measure with its potential $V(x)=|x|$, which is  even and convex. Our upper bound implies that $\alpha_1(n)\le 0$, which suggests that there does not exist a \emph{positive} number $\alpha$ such that the inequality $$\mu_1(\lambda K+(1-\lambda) L)^{\alpha} \geq \lambda \mu_1(K)^{\alpha}+(1-\lambda) \mu_1(L)^{\alpha}$$ holds for all convex bodies $K,L$ in $\mathbb{R}^n$ containing the origin and all $\lambda\in[0,1]$. This exhibits a quite different phenomenon from the case on the class of origin-symmetric convex bodies.
		\vskip 20pt
	
	\section{\bf Asymptotic for the bounds}
    	\vskip 5pt
	
\subsection{Asymptotic for the bounds as \texorpdfstring{$n\to +\infty$}{n -> +infinity}}\label{sub4.1}
\ 
    
	Fix $p\ge 1$. In this part, we give the asymptotic for the bounds in Theorem \ref{main} as $n\to +\infty$, and show that the bounds in Theorem \ref{main} are asymptotically optimal.
	
	\textbf{(1) Asymptotic for the lower bound.}
	
    From Theorem \ref{main}, let
\begin{equation*}
    f_1(n) = \frac{1}{n} e^{\frac{(p-1)n}{p}} ( \frac{(p-1)n}{p} )^{\frac{n}{p}} \Gamma(1-\frac{n}{p}, \frac{(p-1)n}{p}),\quad p\ge 1,\ n=2,3,\ldots .
\end{equation*}
where
\begin{equation*}
    \Gamma(1-\frac{n}{p}, \frac{(p-1)n}{p}) = \int_{\frac{(p-1)n}{p}}^{+\infty} t^{-\frac{n}{p}} e^{-t} dt.
\end{equation*}

Let $t = \frac{(p-1)n}{p} + s$, then
\begin{align}\label{gamma}
    \Gamma(1-\frac{n}{p}, \frac{(p-1)n}{p}) = (\frac{(p-1)n}{p})^{-\frac{n}{p}} e^{-\frac{(p-1)n}{p}} \int_0^{+\infty} (1 + \frac{ps}{(p-1)n})^{-\frac{n}{p}} e^{-s} ds.
\end{align}
Thus,
\begin{equation*}
    f_1(n) = \frac{1}{n} \int_0^{+\infty} (1 + \frac{ps}{(p-1)n})^{-\frac{n}{p}} e^{-s} ds.
\end{equation*}

Since
\begin{align*}
    -\frac{n}{p} \ln(1 + \frac{ps}{(p-1)n}) - s &= -\frac{n}{p} ( \frac{ps}{(p-1)n} - \frac{p^2 s^2}{2(p-1)^2 n^2} + O(n^{-3}) ) - s \\
    &= -\frac{p}{p-1} s + \frac{p s^2}{2(p-1)^2 n} + O(n^{-2}),
\end{align*}
we have
\begin{equation*}
    (1 + \frac{ps}{(p-1)n})^{-\frac{n}{p}} e^{-s} = e^{-\frac{p}{p-1}s} ( 1 + \frac{p s^2}{2(p-1)^2 n} + O(n^{-2})).
\end{equation*}
Consequently,
\begin{equation*}
    \begin{aligned}
        f_1(n) &= \frac{1}{n} \int_0^{+\infty} (1 + \frac{ps}{(p-1)n})^{-\frac{n}{p}} e^{-s} ds \\
        &= \frac{1}{n} \int_0^{+\infty} e^{-\frac{p}{p-1}s} ds + \frac{1}{n} \cdot \frac{p}{2(p-1)^2 n} \int_0^{+\infty} s^2 e^{-\frac{p}{p-1}s} ds + O(n^{-3}) \\
        &= \frac{p-1}{pn} + \frac{p-1}{n^2 p^2} + O(n^{-3}) = \frac{p-1}{pn} + O(n^{-2}).
    \end{aligned}
\end{equation*}

\textbf{(2) Asymptotic for the upper bound.}

    From Theorem \ref{main}, let
\begin{equation*}
    f_2(n) = 1 - \frac{p}{n-1} \frac{\Gamma(\frac{n}{p}) \Gamma(\frac{n+p-2}{p})}{\Gamma(\frac{n-1}{p})^2},\quad p\ge 1,\ n=2,3,\ldots .
\end{equation*}

By the asymptotic expansion of Gamma function (see, e.g., \cite[6.1.40]{handbook})
\begin{equation}\label{spanG}
    \ln \Gamma(x) =(x-\tfrac{1}{2}) \ln x-x+\frac{1}{2} \ln (2 \pi)+\sum_{m=1}^{+\infty} \frac{B_{2 m}}{2 m(2 m-1) x^{2 m-1}}, \quad \text{as}\ x\to +\infty,
\end{equation}
where $B_{2m}$ are the Bernoulli numbers. Especially, $B_2=\frac{1}{6}$. So we have 
\begin{equation*}
    \begin{aligned}
        \ln \Gamma(\frac{n}{p}) - \ln \Gamma(\frac{n-1}{p}) &= \frac{1}{p}\ln(\frac{n}{p}) - \frac{p(p+1)}{2p^2n} + O(n^{-2}),\\
        \ln \Gamma(\frac{n+p-2}{p}) - \ln \Gamma(\frac{n-1}{p}) &= \frac{p-1}{p}\ln(\frac{n}{p}) - \frac{3p(p-1)}{2p^2n} + O(n^{-2}).
    \end{aligned}
\end{equation*}
Hence,
\begin{equation*}
    \begin{aligned}
        \ln \frac{\Gamma(\frac{n}{p}) \Gamma(\frac{n+p-2}{p})}{\Gamma(\frac{n-1}{p})^2}
        &= \big( \ln \Gamma(\frac{n}{p}) - \ln \Gamma(\frac{n-1}{p}) \bigr) + \bigl( \ln \Gamma(\frac{n+p-2}{p}) - \ln \Gamma(\frac{n-1}{p}) \big) \\
        &= \ln(\frac{n}{p}) + \frac{p(1-2p)}{p^2 n} + O(n^{-2}).
    \end{aligned}
\end{equation*}
Therefore,
\begin{equation*}
    \frac{\Gamma(\frac{n}{p}) \Gamma(\frac{n+p-2}{p})}{\Gamma(\frac{n-1}{p})^2} = \frac{n}{p} \exp\bigl( \frac{p(1-2p)}{p^2n} + O(n^{-2}) \bigr) = \frac{n}{p} + \frac{1-2p}{p^2} + O(n^{-1}).
\end{equation*}
Consequently,
\begin{equation*}
    \begin{aligned}
        f_2(n) &= 1 - \frac{p}{n-1} \bigl( \frac{n}{p} + \frac{1-2p}{p^2} + O(n^{-1}) \bigr)= \frac{p-1}{pn} + O(n^{-2}).
    \end{aligned}
\end{equation*}

	Combining the asymptotic of $f_1(n)$ and $f_2(n)$, we have that
	$f_1(n)$ and $f_2(n)$ are asymptotically equivalent as $n \to +\infty$, sharing the exact same leading-order term $\frac{p-1}{pn}$. Thus the estimate in Theorem \ref{main} is asymptotically optimal with respect to \(n\).
	
	\subsection{Asymptotic for the bounds as  \texorpdfstring{$p\to +\infty$}{p -> +infinity}}\label{sub4.2}
	\
    
Fix $n\ge 2$. In this part, we give the asymptotic of  the  bounds in Theorem \ref{main} as $p\to +\infty$.
	
	\textbf{(1) Asymptotic for the  lower bound.}
	
From Theorem \ref{main}, let
	\begin{equation*}
		g_1(p) = \frac{1}{n} e^{\frac{(p-1)n}{p}} ( \frac{(p-1)n}{p} )^{\frac{n}{p}} \Gamma(1-\frac{n}{p}, \frac{(p-1)n}{p}),\quad p\ge 1,\ n=2,3,\ldots .
	\end{equation*}
	By \eqref{gamma}, we have
	$$
	\Gamma(1-\frac{n}{p}, \frac{(p-1)n}{p}) = e^{-\frac{(p-1)n}{p}} (\frac{(p-1)n}{p})^{-\frac{n}{p}} \int_0^{+\infty} (1 + \frac{ps}{n(p-1)})^{-\frac{n}{p}} e^{-s} ds.
	$$
	Thus,
	\begin{equation*}
		g_1(p) = (\frac{(p-1)n}{p})^{-\frac{n}{p}} \int_0^{+\infty} (1 + \frac{ps}{n(p-1)})^{-\frac{n}{p}} e^{-s} ds.
	\end{equation*}
    
	Since
	\begin{equation*}
		\frac{ps}{n(p-1)} = \frac{s}{n} + O(p^{-1}),
	\end{equation*}
	we have 
	\begin{equation*}
		(1 + \frac{ps}{n(p-1)})^{-\frac{n}{p}}= 1 + O(p^{-1}).
	\end{equation*}
	So
	\begin{align*}
		g_1(p)=\frac{1}{n}\int_0^{+\infty} \big(1 +O(p^{-1}) \big)e^{-s}  ds=\frac{1}{n}+O(p^{-1}).
	\end{align*}

    \textbf{(2) Asymptotic for the upper bound.}

    From Theorem \ref{main}, let
\begin{equation*}
    g_2(p) = 1 - \frac{p}{n-1} \frac{\Gamma(\frac{n}{p}) \Gamma(\frac{n+p-2}{p})}{\Gamma(\frac{n-1}{p})^2},\quad p\ge 1,\ n=2,3,\ldots .
\end{equation*}

By the Euler infinite product (see, e.g., \cite[6.1.40]{handbook}), we have
\begin{equation*}
    \begin{aligned}
        \ln \Gamma(x) = -\ln x - \gamma z + O(x^2)\ \   \text{and}\ \
        \ln \Gamma(1+x) = -\gamma x + O(x^2),
    \end{aligned}
\end{equation*}
where $\gamma$ is the Euler constant. Thus, 
\begin{align*}
    \ln \Gamma(\frac{n}{p}) &= -\ln(\frac{n}{p}) - \gamma \frac{n}{p} + O(p^{-2}), \\
    \ln \Gamma(1+\frac{n-2}{p}) &= -\gamma \frac{n-2}{p} + O(p^{-2}), \\
    -2 \ln \Gamma(\frac{n-1}{p}) &= 2\ln(\frac{n-1}{p}) + 2\gamma \frac{n-1}{p} + O(p^{-2}).
\end{align*}
So
\begin{align*}
    \ln \frac{\Gamma(\frac{n}{p}) \Gamma(1+\frac{n-2}{p})}{\Gamma(\frac{n-1}{p})^2}
    &= -\ln(\frac{n}{p}) + 2\ln(\frac{n-1}{p}) + ( -\gamma \frac{n}{p} - \gamma \frac{n-2}{p} + 2\gamma \frac{n-1}{p}) + O(p^{-2}) \\
    &= \ln( \frac{(n-1)^2}{n p} ) + O(p^{-2}).
\end{align*}
Thus
\begin{equation*}
    \frac{\Gamma(\frac{n}{p}) \Gamma(1+\frac{n-2}{p})}{\Gamma(\frac{n-1}{p})^2} = \frac{(n-1)^2}{n p} \big( 1 + O(p^{-2}) \big).
\end{equation*}
That is,
\begin{equation*}
    \frac{\Gamma(\frac{n}{p}) \Gamma(\frac{n+p-2}{p})}{\Gamma(\frac{n-1}{p})^2} = \frac{(n-1)^2}{n p}  + O(p^{-3}) .
\end{equation*}
Therefore,
\begin{align*}
    g_2(p) &= 1 - \frac{p}{n-1} \big( \frac{(n-1)^2}{n p}  + O(p^{-3})  \big) = \frac{1}{n} + O(p^{-2}).
\end{align*}

    It is interesting that both the bounds in Theorem \ref{main} go to $\frac{1}{n}$, and therefore recovers the classical Brunn-Minkowski inequality.

    Indeed, if $K$ and $L$ are two convex bodies containing in $B^n$, then for every $\lambda\in [0,1]$,  $\lambda K+(1-\lambda)L\subset B^n$. Therefore, we have
    $$
    \lim_{p\to +\infty}\frac{\mu_p(K)}{\mu_p(\lambda K+(1-\lambda)L)}=\lim_{p\to +\infty}\frac{\int_K e^{-\frac{|x|^p}{p}}dx}{\int_{\lambda K+(1-\lambda)L} e^{-\frac{|x|^p}{p}}dx}=\frac{V_n(K)}{V_n(\lambda K+(1-\lambda)L)}.
    $$
    Similarly,
    $$
    \lim_{p\to +\infty}\frac{\mu_p(L)}{\mu_p(\lambda K+(1-\lambda)L)}=\frac{V_n(L)}{V_n(\lambda K+(1-\lambda)L)}.
    $$
    If in addition that $K$ and $L$ contains the origin, then
    $$
    1\ge \lambda \big(\frac{\mu_p(K)}{\mu_p(\lambda K+(1-\lambda)L)}\big)^{\alpha_p(n)} +(1-\lambda) \big(\frac{\mu_p(L)}{\mu_p(\lambda K+(1-\lambda)L)}\big)^{\alpha_p(n)}.
    $$
    
    Taking $p\to +\infty$, by the asymptotic for  $\alpha_p(n)$ obtained in this subsection, we have
    $$
    1 \ge \lambda \big( \frac{V_n(K)}{V_n(\lambda K+(1-\lambda)L)} \big)^{\frac{1}{n}} + (1-\lambda) \big( \frac{V_n(L)}{V_n(\lambda K+(1-\lambda)L)}\big)^{\frac{1}{n}}.
    $$
    That is
	$$
	V_n(\lambda K+(1-\lambda) L)^\frac{1}{n} \geq \lambda V_n(K)^\frac{1}{n}+ (1-\lambda) V_n(L)^\frac{1}{n}.
	$$
    Since the Lebesgue measure is affine covariant, this inequality is also valid for general convex bodies $K$ and $L$ in $\mathbb{R}^n$.

	\vskip 20pt
    
	\section{\bf Conclusion Remarks}\label{sec5}

	\vskip 5pt
    
		\subsection{A lower bound for log-concave measures with a \texorpdfstring{$p$}{p}-homogeneous potential} \label{sub5.1}
	\ 

   It is interesting that the proof of lower bound in Theorem \ref{main} can be generalized to   log-concave measures with  $p$-homogeneous potential $V$.
	
	\begin{theorem}\label{qici}
		If $\mu$ is a finite Borel measure on $\mathbb{R}^n$ with density $e^{-V}$, where $V: \mathbb{R}^n \rightarrow[0, +\infty)$ is a $p$-homogeneous convex function for $p\in [1,+\infty)$, and $\alpha$ is the largest number such that 
		$$
		\mu(\lambda K+(1-\lambda) L)^{\alpha} \geq \lambda \mu(K)^{\alpha}+(1-\lambda) \mu(L)^{\alpha}
		$$ holds for all convex bodies $K,L$ in $\mathbb{R}^n$ containing the origin and all $\lambda\in[0,1]$, then
		$$
		\alpha\ge \frac{1}{n}e^{\frac{(p-1)n}{p}} ( \frac{(p-1)n}{p} )^{\frac{n}{p}} \Gamma(1-\frac{n}{p}, \frac{(p-1)n}{p}).
		$$
	\end{theorem}

	\begin{proof}
		It suffices to prove the theorem for $p>1$. Indeed,  since $V$ is a $p$-homogeneous convex function,
        it follows that $\lim_{p\to 1^+}e^{-r^pV(\theta)}=e^{-rV(\theta)}$ and $e^{-rV(\theta)}\le 1$
        for all $r\ge 0$ and $\theta\in \mathbb{S}^{n-1}$.  Hence,  $\lim_{p\to 1^+}\int_{A}e^{-r^pV(\theta)}drd\theta=\int_{A}e^{-rV(\theta)}drd\theta$ for any convex body $A$ in $\mathbb{R}^n$. Similar to Step 3 in  proof of the lower bound in Theorem \ref{main}, we extend the result for $p>1$ to the case $p=1$.
		
		\textbf{Step 1.} Reduce the proof of this Theorem to the case that $\nabla^2V$ is positive definite.
		
	Since $g(x)=|x|^p$ is $p$-homogeneous and $\nabla^2g$ is positive definite for all $x\neq 0$, it follows that  $V+\varepsilon g$ is $p$-homogeneous and  $\nabla^2(V+\varepsilon g)$ is positive definite for all $\varepsilon>0$ almost everywhere. Moreover, $e^{-(V+\varepsilon g)}$ converges to $e^{-V}$ pointwise and is bounded by $1$ on $\mathbb{R}^n$ as $\varepsilon\to 0^+$.
		Let $\mu_\varepsilon$ be the measure with density $e^{-(V+\varepsilon g)}$. Then $\lim_{\varepsilon\to 0^+}\mu_\varepsilon(K)=\mu(K)$ for all $K\in\mathcal{K}^n_o$. Hence, it suffices to prove the theorem for the regularized potential $V+\varepsilon g$.
		
		\textbf{Step 2.} Express the lower bound of $\alpha$.
		
		Assume $V$ is  $p$-homogeneous and $\nabla^2V$ is positive definite almost everywhere. Then we have 
		\begin{equation*}
			\langle\left(\nabla^2 V\right)^{-1} \nabla V, \nabla V\rangle=\frac{p}{p-1} V,
		\end{equation*}
        which is explicitly appeared in \cite[Lemma 3.7]{rotem}. 
		Indeed, by the $p$-homogeneity of $V$, we have
		$$\langle\nabla V, x\rangle=\lim _{t \rightarrow 0} \frac{V(x+t x)-V(x)}{t}=p V(x).$$  Differentiating this equation, it follows that $\nabla^2 V \cdot x= (p-1) \nabla V(x)$. So  $$p V(x)=\langle\nabla V, x\rangle= (p-1)\langle (\nabla^2 V)^{-1}\nabla V, \nabla V\rangle(x).$$ Furthermore, we obtain
		$$
		\frac{n}{n+\langle\left(\nabla^2 V\right)^{-1} \nabla V, \nabla V\rangle }= \frac{n}{n+\frac{p}{p-1}V }.
		$$
        
		In light of $p>1$ and  $V(r\theta)=r^pV(\theta)\ge 0$, it follows that $V(r\theta)$ is  increasing in $r$ for all  $\theta\in\mathbb{S}^{n-1}$. Hence, $\frac{n}{n+\frac{p}{p-1}V(r\theta) }$ is positive and decreasing in $r$ for all  $\theta\in\mathbb{S}^{n-1}$. By the  convexity of $V$, it follows that $\frac{n}{n+\frac{p}{p-1}V(r\theta) }$ is continuous. Hence, using Lemma \ref{gen} and Lemma \ref{dandiao}, we have
		$$
		\begin{aligned}
			\alpha\ge&\inf_{K \in \mathcal{K}^n_o} \frac{1}{n\mu(K)}\int_K \frac{n}{n+\langle\left(\nabla^2 V\right)^{-1} \nabla V, \nabla V\rangle }d \mu\\=&\frac{1}{n}\inf_{\theta\in \mathbb{S}^{n-1}}\{\frac{\int_{0}^{+\infty}\frac{n}{n+\frac{p}{p-1}V(r\theta)}r^{n-1}e^{-V(r\theta)}dr}{\int_{0}^{+\infty}r^{n-1}e^{-V(r\theta)}dr}\}\\
			=&\frac{1}{n}\inf_{\theta\in \mathbb{S}^{n-1}}\{\frac{\int_{0}^{+\infty}\frac{ne^{-r^pV(\theta)}}{n+\frac{p}{p-1}r^pV(\theta)}r^{n-1}dr}{\int_{0}^{+\infty}e^{-r^pV(\theta)}r^{n-1}dr}\}.
		\end{aligned}
		$$
		
		\textbf{Step 3:} Compute the lower bound of $\alpha$.
		
		Define $$g(t)=\frac{\int_0^{+\infty} \frac{n e^{-tr^p}}{n+\frac{p}{p-1} t r^p} r^{n-1} d r}{\int_0^{+\infty} e^{-tr^p} r^{n-1} d r},\quad t>0.$$
		Let $u=tr^p$, $t>0$, $r\ge 0$. Then
		$$
		\begin{aligned}
			\int_0^{+\infty} \frac{n e^{-tr^p}}{n+\frac{p}{p-1} t r^p} r^{n-1} d r=&\int_0^{+\infty} \frac{n e^{-u}}{n+\frac{p}{p-1} u} \cdot(\frac{u}{t})^{(n-1) / p} \cdot \frac{1}{p} t^{-1 / p} u^{1 / p-1} d u 
		\\=&\frac{n}{p} t^{-n / p} \int_0^{+\infty} \frac{e^{-u} u^{n / p-1}}{n+\frac{p}{p-1} u} d u
		\end{aligned}
		$$
		and
        $$\begin{aligned} \int_0^{+\infty} e^{-tr^p} r^{n-1} d r=&\int_0^{+\infty} e^{-u} \cdot(\frac{u}{t})^{(n-1) / p} \cdot \frac{1}{p} t^{-1 / p} u^{1 / p-1} d u \\  =&\frac{1}{p} t^{-n / p} \int_0^{+\infty} e^{-u} u^{n / p-1} d=\frac{1}{p} t^{-n / p} \Gamma(\frac{n}{p}).\end{aligned}$$
        
		Thus, 
		\begin{equation}\label{gt}
			g(t)=\frac{\frac{n}{p} t^{-n / p} \int_0^{+\infty} \frac{e^{-u} u^{n / p-1}}{n+\frac{p}{p-1} u} d u}{\frac{1}{p} t^{-n / p} \Gamma(\frac{n}{p})}=\frac{n}{\Gamma(\frac{n}{p})} \int_0^{+\infty} \frac{e^{-u} u^{\frac{n}{p}-1}}{n+\frac{p}{p-1} u} d u
		\end{equation}
		 is \emph{independent} of $t$. Similar to Step 2 in the proof of  lower bound in Theorem \ref{main}, we have  
		$$
		g(t)\equiv\frac{1}{n}e^{\frac{(p-1)n}{p}} ( \frac{(p-1)n}{p} )^{\frac{n}{p}} \Gamma(1-\frac{n}{p}, \frac{(p-1)n}{p}), \quad t>0.
		$$
        
		Consequently, 
		$$
		\begin{aligned}
			\alpha\ge&\frac{1}{n}\inf_{\theta\in \mathbb{S}^{n-1}}\{g(V(\theta))\}=\frac{1}{n}e^{\frac{(p-1)n}{p}} ( \frac{(p-1)n}{p} )^{\frac{n}{p}} \Gamma(1-\frac{n}{p}, \frac{(p-1)n}{p}).
		\end{aligned}
		$$
		This completes the proof.
	\end{proof}
	
	\begin{remark}\label{re2}
     In Theorem 1.3 of Aishwarya and Rotem \cite{rotem}, they showed  the exponent $\alpha\ge \frac{p-1}{pn}$ on star bodies in $\mathbb{R}^n$. It is naturally that $\alpha\ge \frac{p-1}{pn}$ on  convex bodies containing the origin in $\mathbb{R}^n$. 
     
     It is interesting that the lower bound  in Theorem \ref{qici} satisfies that
		$$
		\frac{1}{n}g(t)=\frac{1}{\Gamma(\frac{n}{p})} \int_0^{+\infty} \frac{e^{-u} u^{\frac{n}{p}-1}}{n+\frac{p}{p-1} u} d u>\frac{p-1}{pn}.
		$$
        
        Indeed, let $f(u)=\frac{1}{n+\frac{p}{p-1} u}$ and $w(u)=\frac{e^{-u} u^{n / p-1}}{\Gamma(\frac{n}{p})}$, $u>0$. Then $f$ is strictly convex in $\mathbb{R}$, and $\int_{0}^{+\infty}w(u)du=1$. By the Jensen inequality, we have
		$$
		\begin{aligned}
			\frac{1}{\Gamma(\frac{n}{p})} \int_0^{+\infty} \frac{e^{-u} u^{\frac{n}{p}-1}}{n+\frac{p}{p-1} u} d u&=\int_0^{+\infty} f(u) w(u) d u>f(\int_0^{+\infty} u w(u) d u)\\
			&=f(\frac{\Gamma(\frac{n}{p}+1)}{\Gamma(\frac{n}{p})})=f(\frac{n}{p})=\frac{p-1}{pn}.
		\end{aligned}
		$$

        \end{remark}

        \subsection{Illustration of the bounds}\label{sub5.2}
\ 

        In the following, we illustrate our bounds and the known bounds by Python 3.14.0.

        For $\alpha_2(n)$, we illustrate our bounds in Theorem \ref{main}, Kolesnikov-Livshyts's \cite{Livshyts}  lower bound $\frac{1}{2n}$, and the trivial upper bound $\frac{1}{n}$ in Figure \ref{fig1}. 
        The left picture is plotted in the standard Cartesian coordinates and the right picture is plotted in the log‑log coordinates.
     \begin{figure}[ht]
  \centering
  \includegraphics[width=0.85\linewidth]{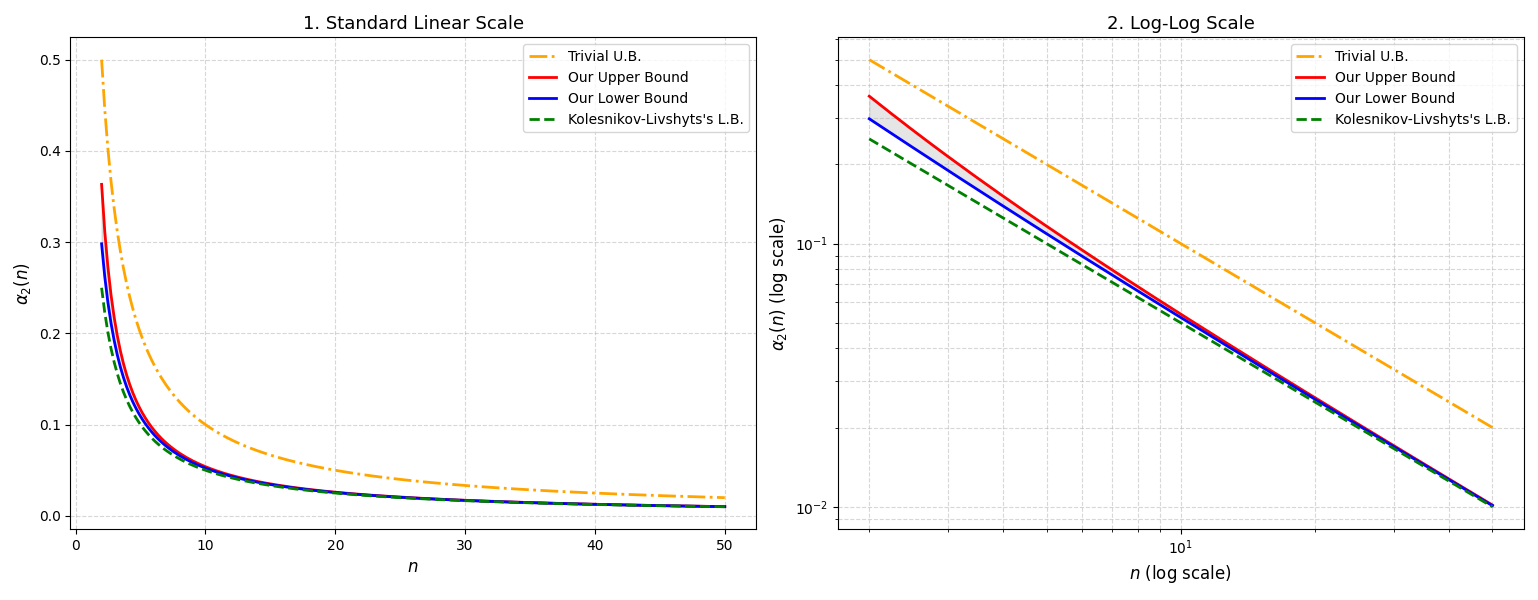}
  \caption{}
  \label{fig1}
\end{figure}

For $\alpha_p(2),\alpha_p(3)$ and $\alpha_p(10)$, we illustrate our bounds in Theorem \ref{main}, Aishwarya-Rotem's \cite{rotem} lower bound $\frac{p-1}{pn}$, and the trivial upper bound $\frac{1}{n}$ as \(p\) varies in Figure \ref{fig2}.
\begin{figure}[ht]
  \centering
  \includegraphics[width=0.9\linewidth]{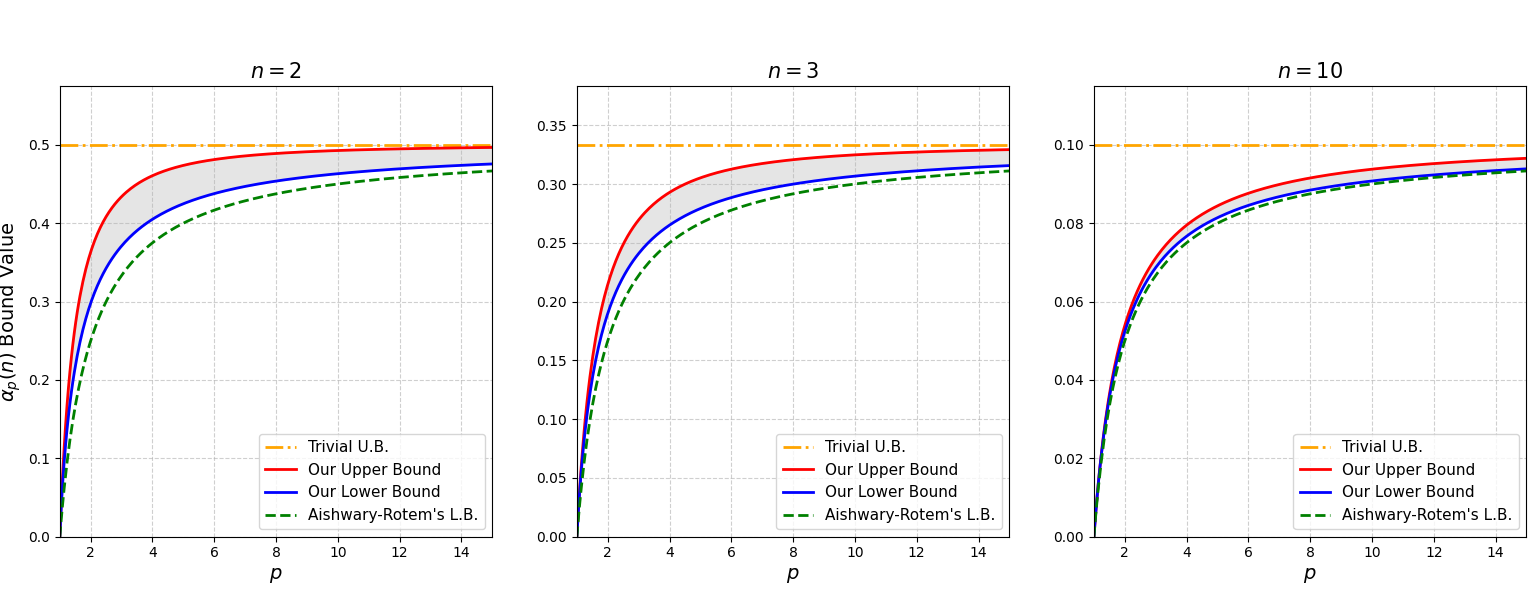}
\caption{}
  \label{fig2}
\end{figure}

The results shown in the figures are consistent with our asymptotic analysis.

	\vskip3pt {\bf Conflict of Interest}: We declare that we have no conflict of interest.
	\vskip3pt {\bf Data Availability}: Not applicable.

\end{document}